\documentclass[12pt]{amsart}
\usepackage{amscd, amssymb, amsmath}
\usepackage{color}

\title{On selfinjective algebras of stable dimension zero}
\author[M.~Yoshiwaki]{Michio Yoshiwaki}
\address{Department~of~Mathematics~and~Physics, 
Graduate~School~of~Science,
Osaka~City~University,
3-3-138 Sugimoto, Sumiyoshi-ku, Osaka~558-8585, Japan}
\email{yosiwaki@sci.osaka-cu.ac.jp}


\newtheorem{conjecture}{Conjecture}
\newtheorem*{mthm}{Main Theorem}
\newtheorem*{cor}{Corollary }
\newtheorem{thmenum}{Theorem}[section]
\newtheorem{prpenum}[thmenum]{Proposition}
\newtheorem{lemenum}[thmenum]{Lemma}
\newtheorem{corenum}[thmenum]{Corollary}

\theoremstyle{definition}
\newtheorem{dfnenum}[thmenum]{Definition}

\theoremstyle{remark}
\newtheorem{rmkenum}[thmenum]{Remark}

\newcommand{\umod}{\operatorname{\underline{mod}}\nolimits}

\def\LL{\operatorname{LL}}
\def\d{\operatorname{d}}
\def\dim{\operatorname{dim}}
\def\stab{\operatorname{stab}}
\def\rep{\operatorname{rep}}
\def\End{\operatorname{End}}
\def\add{\operatorname{add}}

\def\mod{\operatorname{mod}}
\def\Hom{\operatorname{Hom}}
\def\Tr{\operatorname{Tr}}
\def\gl{\operatorname{gl}}
\def\rad{\operatorname{rad}}
\def\Hom{\operatorname{Hom}}
\def\blank{\operatorname{-}}
\def\P{\operatorname{P}}

\def\op{^{\text{\rm op}}}
\def\iso{\cong}
\def\niso{\not\cong}
\def\ra{\rightarrow}
\def\Integer{{\mathbb Z}}

\def\Tm{{\mathcal T}}
\def\Im{{\mathcal I}}
\def\Jm{{\mathcal J}}

\begin{document}
\begin{abstract}
Let $A$ be a selfinjective algebra over an algebraically closed field. 
We study the stable dimension of $A$, which is the dimension of the stable module category of $A$ in the sense of Rouquier. 
Then we prove that 
$A$ is representation-finite 
if the stable dimension of $A$ is $0$. 
\end{abstract}

\maketitle

\section{Introduction}
R. Rouquier introduced a notion of dimension of a triangulated category in \cite{R1}. 
One of his aims was to give a lower bound for Auslander's representation dimension of selfinjective algebras;
namely, he showed that for any non-semisimple selfinjective finite-dimensional algebra $A$ over a field $k$, 
the representation dimension of $A$ is greater than or equal to 
the stable dimension of $A$ plus two (see \cite{R2}). 

On the other hand, M. Auslander introduced the representation dimension in \cite{Au}, 
and hoped that the representation dimension should be a good measure of
how far a representation-infinite algebra is from being representation-finite.  
Actually, he showed that for any artin algebra $\Lambda$, $\Lambda$ is representation-finite if and only if it has representation dimension at most two.  

Thus,  
any non-semisimple representation-finite selfinjective finite-dimensional algebra over a field $k$ must have stable dimension zero, which also follows from definition immediately.
Then a natural question arises as to whether the converse should also hold.

The aim of this paper is to prove that the converse holds provided that $k$ is algebraically closed. 
Namely, our main theorem (see Theorem~\ref{Y1}) is the following.
\begin{mthm} \label{Y2}
Let $A$ be a non-semisimple selfinjective finite-dimensional connected algebra over an algebraically closed field $k$.  
If the set ${}_s \Gamma(A)_0$ of isoclasses of non-projective indecomposable $A$-modules admits only finitely many $\Omega$-orbits, 
then $A$ is representation-finite.
\end{mthm}
As a corollary, we have the desired result (see Corollary~\ref{cor1}). 
\begin{cor}
If the stable dimension of $A$ is zero, then $A$ is representation-finite. 
\end{cor}

Although this was expected to hold by some experts, it had not been proved before. 
\section{Stable dimension}
Let $\Tm$ be a triangulated category with shift functor $\operatorname{T}$. 
And let $\Im$ and $\Jm$ be full subcategories of $\Tm$. 
We denote by $\langle\Im\rangle$ the smallest full subcategory of $\Tm$ containing $\Im$ and closed
under taking shifts and direct summands of finite direct sums. 
Moreover, we denote by $\Im *\Jm$ the full subcategory of $\Tm$ consisting of objects $M\in\Tm$ 
for which there exists a triangle $I \ra M \ra J \ra \operatorname{T}I$ with $I \in\Im$ and $J \in \Jm$. 
Put $\Im \diamond \Jm = \langle \Im * \Jm\rangle$, and
inductively we define 
$$ \langle \Im \rangle_{n+1} = \left\{ \begin{array}{ll}
\langle \Im \rangle, & \text{if } n=0;\\
\langle \Im \rangle_{n} \diamond \langle \Im \rangle & \text{if } n\geq 1.
\end{array} 
\right.
$$ 
Then the {\it dimension} of $\Tm$ is
$$\dim \Tm:=\min \{ n\geq 0 \ |\ \langle M \rangle_{n+1}=\Tm\ \text{for some }M \},$$ 
or $\infty$ when there is no such an object $M$ for any $n$.

Let $A$ be a non-semisimple selfinjective finite-dimensional algebra over a field $k$. 
We denote by $A\blank\mod$ the abelian category of finite dimensional (over $k$) left $A$-modules, 
and by $A\blank\umod$ the stable module category of $A$, 
whose objects are the same as those of $A\blank\mod$ 
and morphisms from $X$ to $Y$ are given by the quotient space $\Hom_A (X,Y)/\P(X,Y)$, 
where $\P(X,Y)$ consists of those morphisms from $X$ to $Y$ which factor through a projective $A$-module. 
The syzygy functor $\Omega : A\blank\umod \ra A\blank\umod$ is a functor 
defined by the correspondence of an object $X$ to the kernel of the projective cover of $X$. 
Note that if $X$ is indecomposable, then $\Omega(X)$ is indecomposable. 
Moreover, note that the syzygy functor $\Omega$ is an equivalence, 
and the stable module category $A\blank\umod$ is a triangulated category 
with shift functor $\Omega^{-1}$ (see Happel \cite{Ha}). 
Then we define the stable dimension of selfinjective algebras. 

\begin{dfnenum}
The {\it stable dimension} of $A$ is  
$$\stab.\dim A:=\dim (A\blank\umod).$$
\end{dfnenum}

\begin{rmkenum}
Let $A$ and $B$ be non-semisimple selfinjective finite-dimensional algebras over a field $k$.
If $A$ and $B$ are {\it stably equivalent} as triangulated categories 
(that is, $A\blank\umod$ and $B\blank\umod$ are equivalent as triangulated categories), 
then $\stab.\dim A$ and $\stab.\dim B$ coincide:  
for example, if $A$ and $B$ are derived equivalent, 
then they are stably equivalent as triangulated categories (Rickard \cite{Ri}), 
and hence $\stab.\dim A$ and $\stab.\dim B$ coincide. 
Note that if $k$ is an algebraically closed field and if either $A$ or $B$ is representation-finite, 
they are derived equivalent if and only if they are stably equivalent as triangulated categories (Asashiba \cite{Asa}).  
\end{rmkenum}

The {\it representation dimension} of a non-semisimple artin algebra $\Lambda$ is 
$$\rep.\dim \Lambda:=\min\{\gl.\dim \End_{\Lambda}(M)\ |\ M:\text{a generator and a cogenerator}\}.$$
(For a semisimple artin algebra, it is defined to be $1$.)

In \cite{R2}, Rouquier has shown the following result. 
\begin{prpenum}[Rouquier \cite{R2} {\rm cf}. Auslander \cite{Au}] \label{Rou1}
Let $A$ be a non-semisimple selfinjective finite-dimensional algebra over a field. 
Then
$$\LL (A)\geq \rep.\dim A \geq \stab.\dim A +2 $$
where 
the Loewy length $\LL (A)$ is the smallest integer $r$ such that $\rad (A)^r = 0$. 
\end{prpenum}
After Auslander proved in \cite{Au} (see Proposition p.55) that $\LL (A)+1 \geq \rep.\dim A$,  
Rouquier has improved it by indicating that the equality does not occur,  
and hence the first inequality in Proposition~\ref{Rou1}. 
\begin{rmkenum}
The stable dimension is always finite by the first inequality in Proposition~\ref{Rou1}. 
Recall also that for any artin algebra, the representation dimension is always finite (see Iyama~\cite{I}). 
\end{rmkenum}

In \cite{Au}, Auslander has also shown the following result.
\begin{thmenum}[Auslander \cite{Au}] \label{Au1}
Let $\Lambda$ be an artin algebra. Then $\Lambda$ is representation-finite if and only if $\rep.\dim \Lambda\leq 2$.
\end{thmenum} 
Thus we observe that 
any (non-semisimple) representation-finite selfinjective finite-dimensional algebra over a field $k$ has stable dimension $0$.
It has been believed that the converse also holds, 
but this is non-trivial.
We, therefore, give an argument to answer affirmatively to this in the case that $k$ is an algebraically closed field
\section{Main results}
Throughout the rest,  
$k$ denotes a fixed algebraically closed field, 
and all algebras are finite-dimensional associative $k$-algebras with an identity. 

For any $k$-algebra $A$,
we denote by $\Gamma(A)$ the Auslander-Reiten quiver of $A$ 
whose vertices are isoclasses of indecomposable (left) $A$-modules and arrows are irreducible maps. 
We may identify the vertices of $\Gamma(A)$ with the indecomposable $A$-modules. 
Then we have the Auslander-Reiten translation $\tau= D\Tr$ and $\tau^{-1}=\Tr D$
where $D:A\blank\mod \ra A\op\blank\mod$ is the standard duality $\Hom_k (\blank,k)$. 

In the sequel, we assume that 
$A$ is a non-semisimple selfinjective connected algebra, unless otherwise stated.
Thus we have $\tau\iso\Omega^2\nu\iso\nu\Omega^2$ since $\Omega\nu\iso\nu\Omega$,  
where $\nu=D\Hom_A(\blank,A)$ is the Nakayama functor. 

In addition, we denote by ${}_s \Gamma(A)$ the stable Auslander-Reiten quiver of $A$, 
obtained from $\Gamma(A)$ by removing the projective-injective vertices and the arrows attached to them. 
Then the set ${}_s\Gamma(A)_0$ of vertices of ${}_s\Gamma(A)$ coincides with the set of isoclasses of non-projective indecomposable $A$-modules. 
It is well-known that we can recover $\Gamma(A)$ from ${}_s \Gamma(A)$. 
For a connected component $C$ of $\Gamma(A)$, we denote by ${}_s C$ the stable part of $C$. 
Note that the Auslander-Reiten translation $\tau$ is an automorphism of the quivers ${}_s \Gamma(A)$ and ${}_s C$, and its inverse $\tau^{-1}$. 

Then we mention the following fundamental fact on selfinjective algebras of stable dimension $0$.  

\begin{lemenum} \label{lem1}
The following are equivalent$:$ 
\begin{enumerate}
\item $\stab.\dim A = 0$, 
\item ${}_s \Gamma(A)_0$ admits only finitely many $\Omega$-orbits.
\end{enumerate}
\end{lemenum}
\begin{proof}
By definition, the assertion follows that $\stab.\dim A = 0$ implies $A\blank\umod=\langle M\rangle=\add\{\Omega^i M| i\in\Integer \}$ for some $M\in A\blank\umod$. 
\end{proof}

An indecomposable $A$-module $M$ is called {\it $\tau$-periodic} if $\tau^n M \iso M$ for some $n>0$. 
Then it is easy to prove the following lemma on a connected component of $\Gamma(A)$ containing a $\tau$-periodic module. 
\begin{lemenum} \label{lem2}
Suppose that a connected component $C$ of $\Gamma(A)$ contains a $\tau$-periodic module.
Then all modules in ${}_s C$ are $\tau$-periodic. 
\end{lemenum}

Given an indecomposable $A$-module $M$, we denote by $\tau^{\Integer}M$ the $\tau$-orbit of $M$, 
that is, $\tau^{\Integer} M:=\{\tau^n M\ |\ n\in\Integer\}$. 
Similarly, the $\Omega$-orbit of $M$ and the $\langle\tau,\ \Omega\rangle$-orbit of $M$ 
are denoted by $\Omega^{\Integer}M:=\{\Omega^m M\ |\ m\in\Integer\}$ and
$\tau^{\Integer}\Omega^{\Integer}M:=\{\tau^n\Omega^m M\ |\ n,m\in\Integer\}$, respectively. 

For a set $S$ of indecomposable $A$-modules, 
we denote by $\d(S)$ the supremum of $k$-dimensions of modules in $S$,  
namely, $$\textstyle\d(S):=\sup_{X\in S} (\dim_k X).$$ 
Then 
\begin{lemenum} \label{lem3}
If an indecomposable $A$-module $M$ is $\tau$-periodic, then $\d(\Omega^{\Integer}M)<\infty.$
\end{lemenum}
\begin{proof}
By the hypothesis, there exists $n>0$ such that $\tau^n M \iso M$. 
Since $\tau \iso \Omega^2\nu$, 
we have $\nu^{n}M\iso\Omega^{-2n}M$.   
Since the Nakayama functor $\nu$ preserves the dimension of modules, 
$\dim_k \Omega^i M$ is determined by $i$ modulo $2n$. 
The assertion follows. 
\end{proof}

Therefore, we have the following fact.  
\begin{lemenum} \label{lem4}
If ${}_s \Gamma(A)_0$ admits only finitely many $\Omega$-orbits,
then $$\d(\{\text{all } \tau\text{-periodic modules in } {}_s \Gamma(A)\})<\infty.$$ 
\end{lemenum}
\begin{proof}
Let $M_1,\cdots,M_l$ be a complete list of representatives of $\Omega$-orbits of $\tau$-periodic modules in ${}_s \Gamma(A)_0$.
Note that if $M$ is $\tau$-periodic, then all modules in $\Omega^{\Integer}M$ are $\tau$-periodic.
Then $\cup_{i} \Omega^{\Integer}M_i$ is just the set of all $\tau$-periodic modules in ${}_s \Gamma(A)$.   
By Lemma~\ref{lem3} we have 
$$
\begin{array}{rcl}
\d(\{\text{all } \tau\text{-periodic modules in } {}_s \Gamma(A)\})&=&\d(\cup_{i} \Omega^{\Integer}M_i)\\
&=&\sup_{i} \d(\Omega^{\Integer}M_i) <\infty,
\end{array}
$$
as desired. 
\end{proof}

Now, we mention the following well-known fact on representation-finite algebras due to Auslander. 
\begin{lemenum}[Auslander \cite{AMS}] \label{Au2}
Let $A$ be a finite-dimensional connected $k$-algebra,  
and let $C$ be a connected component of $\Gamma(A)$. 
If $\d(C)<\infty$, then $A$ is representation-finite.
\end{lemenum}

By using this lemma, we prove the following key proposition. 

\begin{prpenum} \label{keyprp}
Assume that ${}_s \Gamma(A)_0$ admits only finitely many $\Omega$-orbits.
If there exists a $\tau$-periodic module in ${}_s \Gamma(A)$, then $A$ is representation-finite.
\end{prpenum}
\begin{proof}
Let $X$ be a $\tau$-periodic module in ${}_s \Gamma(A)$, 
and let $C$ be a connected component of $\Gamma(A)$ containing $X$.  
By Lemma~\ref{lem2}, all modules in ${}_s C$ are $\tau$-periodic. 
Note that there exist only finitely many projective-injective vertices.
Then by Lemma~\ref{lem4}, $\d(C)<\infty$,  
and thus $A$ is representation-finite by Lemma~\ref{Au2}. 
\end{proof}

Here, we mention the following fact on non-$\tau$-periodic modules.
\begin{lemenum} \label{lem5}
Let $M$ be an indecomposable $A$-module.  
Assume that there exists a positive integer $n$ such that $\nu^n M \in \Omega^{\Integer} M$.  
If $M$ is not $\tau$-periodic, 
then $\tau^{\Integer}\Omega^{\Integer}M$ is a union of finite number of $\tau$-orbits.  
\end{lemenum}
\begin{proof}
Let $m$ be an integer such that $\nu^{n} M \iso \Omega^{m}M$. 
By the hypothesis we have 
$$\tau^{n} M \iso \Omega^{2n}\nu^{n} M \iso \Omega^{2n+m} M (\niso M)$$ with $2n+m\neq 0$. 
Put $t:=|2n+m|$.  
Then $$\Omega^{\Integer}M\subset\tau^{n\Integer}M\cup\tau^{n\Integer}\Omega M\cup\cdots\cup\tau^{n\Integer}\Omega^{t-1}M,$$ 
and thus we have 
$$\tau^{\Integer}\Omega^{\Integer}M=\tau^{\Integer}M\cup\tau^{\Integer}\Omega M\cup\cdots\cup\tau^{\Integer}\Omega^{t-1}M,$$ 
as desired.
\end{proof}

Moreover, we need Liu's characterization of representation-finite algebras over an algebraically closed field to prove our main theorem.  

\begin{lemenum}[Liu \cite{Liu} 3.11.Proposition p.52] \label{Liu}
Let $A$ be a finite-dimensional algebra over an algebraically closed field $k$. 
Then $A$ is representation-finite 
if and only if $\Gamma(A)$ admits only finitely many $\tau$-orbits.   
\end{lemenum}

This follows from the second Brauer-Thrall conjecture. 
Therefore, the assumption that the base field $k$ is algebraically closed is essential. \\

Now we prove our main theorem. 

\begin{thmenum} \label{Y1}
Let $A$ be a non-semisimple selfinjective finite-dimensional connected algebra over an algebraically closed field $k$.  
If the set ${}_s \Gamma(A)_0$ of isoclasses of non-projective indecomposable $A$-modules admits only finitely many $\Omega$-orbits, 
then $A$ is representation-finite.
\end{thmenum}

\begin{proof}
Let $M_1,\cdots,M_l$ be a complete list of representatives of $\Omega$-orbits in ${}_s \Gamma(A)_0$. 
Note that for each $i \in \{1,\cdots,l\}$ 
there exists a positive integer $n_i$ such that $\nu^{n_i} M_i \in \Omega^{\Integer} M_i$. 

If there exists an integer $s \in \{1,\cdots,l\}$ 
such that $M_s$ is $\tau$-periodic,  
then by Proposition~\ref{keyprp}, $A$ is representation-finite. 

Then we can assume that each $M_i$ is not $\tau$-periodic.  
By Lemma~\ref{lem5}, $\Gamma(A)$ admits only finitely many $\tau$-orbits. 
Thus $A$ is representation-finite by Lemma~\ref{Liu}, as desired.   
\end{proof}

By Lemma \ref{lem1}, we obtain the following result.  

\begin{corenum} \label{cor1}
If the stable dimension of $A$ is zero, then $A$ is representation-finite. 
\end{corenum}

\begin{proof}
Suppose that $\stab.\dim A=0$. 
Then by Lemma \ref{lem1}, ${}_s \Gamma(A)_0$ admits only finitely many $\Omega$-orbits,   
and thus $A$ is representation-finite by Theorem~\ref{Y1}.    
\end{proof}

Furthermore, we obtain the following result. 

\begin{corenum} \label{cor2}
If $\rep.\dim A=3$, then $\stab.\dim A=1$. 
\end{corenum}
\begin{proof}
If $\rep.\dim A=3$, then $\stab.\dim A\leq 1$ and $A$ is not representation-finite by Proposition \ref{Rou1} and Theorem \ref{Au1}.
Then by Corollary~\ref{cor1}, $\stab.\dim A=1$. 
\end{proof}

It has been conjectured whether the following holds.

\begin{conjecture} \label{conj1}
Any artin algebra of tame representation type has representation dimension at most $3$.
\end{conjecture}

Moreover, we have a new conjecture by Corollary~\ref{cor2}.

\begin{conjecture} \label{conj2}
Any $($non-semisimple$)$ selfinjective $k$-algebra of tame representation type has stable dimension at most $1$. 
\end{conjecture}

Conjecture~\ref{conj1} is true for some classes of tame algebras, 
such as domestic selfinjective algebras socle equivalent to a weakly symmetric algebra of Euclidean type (see Bocian-Holm-Skowro\'nski \cite{BHS}).
Thus, these algebras have stable dimension at most $1$ by Proposition~\ref{Rou1}.  

Note that in general the converse of Conjecture~\ref{conj1} does not hold. 
Any wild hereditary artin algebra is a counter example.

Similar to Conjecture~\ref{conj1}, the converse of Conjecture~\ref{conj2} does not hold in general: 
for instance, any wild selfinjective finite-dimensional algebra over a field with radical cube zero has stable dimension at most $1$ (see Proposition~\ref{Rou1}).

\section*{Acknowledgments}
The author would like to thank Takuma Aihara for suggesting him this problem. 
Also, 
the author would like to express his sincere gratitude to Osamu Iyama for helpful discussions and comments, 
and Hideto Asashiba for helpful advice.

\end{document}